\documentclass[12pt, twoside]{article}
\usepackage{a4, amssymb,amsmath,amsthm,latexsym, cite, hyperref, graphicx,amsfonts, ulem,float}
\usepackage{tikz, tikz-network, pgfplots, authblk} 
\usetikzlibrary{shapes, arrows, positioning}
\usepackage[font=small,labelsep=none,figurename=Fig.]{caption}
\newcommand{\bigzero}{\mbox{\normalfont\Large\bfseries 0}}
\newcommand{\rvline}{\hspace*{-\arraycolsep}\vline\hspace*{-\arraycolsep}}
\newtheorem{theorem}{Theorem}[section]

\newtheorem{definition}[theorem]{Definition}
\newtheorem{example}[theorem]{Example}
\newtheorem{lemma} [theorem]{Lemma}

\newtheorem{observation}[theorem]{Observation}
\newtheorem{remark}[theorem]{Remark}

\vspace{5cm}

\begin{document}
\title{\bf Laplacian Spectra of Semigraphs}
\author{Pralhad M. Shinde \footnote{Department of Mathematics, College of Engineering Pune, Pune - 411005, India,
 \\ Email: pralhad.shinde96@gmail.com}} 
\date{}
\maketitle
\thispagestyle{empty}

\begin{abstract}	
{\footnotesize  Consider a semigraph $G=(V,\,E)$; in this paper, we study the eigenvalues of the Laplacian matrix of $G$. We show that the Laplacian of $G$ is positive semi-definite, and $G$ is connected if and only if $\lambda_2 >0.$ Along the similar lines of graph theory bounds on the largest eigenvalue, we obtain upper and lower bounds on the largest Laplacian eigenvalue of G and enumerate the Laplacian eigenvalues of some special semigraphs such as star semigraph, rooted 3-uniform semigraph tree. }
 \end{abstract}



\section{Introduction}
In \cite{smt}, Sampathkumar generalized the definition of a graph and defined the notion of a semigraph. A semigraph structure looks similar to a linear hypergraph structure but has ordered edges; hence, semigraphs are different from linear hypergraphs.  The adjacency matrix of semigraphs is considered in \cite{cmd}, \cite{smt}, but the matrix is not symmetric. In an attempt to refine the definition and make it symmetric so that one can use linear algebra tools, the adjacency matrix of semigraphs is introduced and studied in \cite{pms1}. The adjacency matrix of a semigraph is symmetric, and when a semigraph is a graph it gives us the adjacency matrix of a graph. In \cite{pms1}, author showed that the spectral graph theory could be extended to the spectral semigraph theory.  This motivated us to look at the Laplacian matrix of semigraphs. We introduce the Laplacian matrix of a semigraph and investigate its spectra along similar lines to the Laplacian spectra of graphs. The generalization is natural because when a semigraph is a graph, our Laplacian matrix coincides with the Laplacian matrix of graphs \cite{rbp}. This paper focuses on studying the spectral properties of the semigraph Laplacian. In section 2, we give the definitions of the Laplacian and signless Laplacian matrices of a semigraph and show that both are positive semi-definite. Further, we show that semigraph is connected iff the second Laplacian eigenvalue $\lambda_2 >0$. In section 3, we obtain upper and lower bounds on the largest Laplacian eigenvalue $\lambda_n$ of a semigraph. In section 4, we enumerate Laplacian spectra of some special types of semigraphs: star semigraph, rooted 3-uniform semigraph tree. 

\section*{Preliminaries}
In this section, we recall some definitions, and for all other standard notations and definitions, we refer to \cite{pms1}, \cite{smt}.
\begin{definition} \label{def:1}
Consider a non-empty set $V$ with $|V|=n\geq 2$, $P_{o}(V)$ denotes the set of all ordered $k$-tuples of elements of $V,$ $1\leq k\leq n$.  A pair G=$(V, E),$ where $V$ is called vertex set and $E\subseteq P_{o}(V)$ is called edge set, defines a semigraph if it satisfies the following two conditions:

\begin{enumerate}
	\item  For all $e_1, e_2 \in E$, $|e_1 \cap e_2|\leq 1$.
	\item Let $e_1=(v_1,v_2,\cdots,v_p)$ and $e_2=(u_1,u_2,\cdots,u_q)$ be two edges, $e_1=e_2$ if 
	\begin{enumerate}
	\item[i)] p = q and
	\item[ii)] either $v_k = u_k$ for $1\leq k \leq p,$  or $v_k=u_{p-k+1}$, for $1\leq k \leq p.$ 
	\end{enumerate}
\end{enumerate}
\end{definition}

\noindent Note that the edges $(u_1,u_2,\cdots,u_r)$ and $(u_r,u_{r-1},\cdots,u_1)$ are equal.\\\\
Let $V$ and  $E$ be vertex and edge sets of a semigraph $G$, the vertices $v_i,\; v_j \in V$ are said to be {\bf\itshape adjacent\;} if $\{v_i,\;v_j\} \subseteq e$ for some edge $e\in E$. If $v_i, \;v_j$ are two vertices which are consecutive in order such that  $\{v_i,\;v_j\} \subseteq e$ for some edge $e\in E$ then we say that they are {\bf\itshape consecutively adjacent}. 
\vskip2mm
\noindent Let $e=(v_1,v_2,\cdots,v_p)$ be an edge, the vertices $v_1$, $v_p$ of an edge $e$ are {\bf \it end} vertices, and $v_2,v_3,\cdots,v_{p-1}$ are {\bf \it middle} vertices. Here, for all $1\leq i, j\leq p$,  vertices $v_i, \; v_j$ are adjacent while $v_i, v_{i+1}$ are consecutively adjacent $\forall \,1\leq i\leq p-1$. 
The ordering in the edges gives rise to different types of vertices and edges, and we define those as follows: If $v_i$ is an end vertex $\forall \, e$ with $v_{i}\in e$, then $v_i$ is called a {\bf \it pure end vertex}, and if $v_{i}$ is middle vertex $\forall \, e$ with $v_{i}\in e$ then $v_i$ is called a {\bf \it pure middle vertex}. If for $v_i$, there exist $e_1, e_2\in E$ such that $v_i$ is a middle vertex of $e_1$ and end vertex of $e_2$ then $v_i$ is called a {\bf \it middle end vertex}. If both $v_1$ and $v_r$ of an edge $e=(v_1, v_2, \cdots, v_r),\; r\geq2$ are pure end vertices then $e$ is called {\bf \it full edge}. If either $v_1$ or $v_r$ (or both) of an edge $e=(v_1, v_2, \cdots, v_r),\; r>2$  are middle end vertices or if $e=(v_1, v_2)$ and exactly one of $v_1, v_2$ is a pure end vertex and the other is a middle end vertex then $e$ is called an {\bf \it half edge}. If both  $v_1$ and $v_2$ of an edge $e=(v_1, v_2)$ are middle end vertices then we say that $e$ is a {{\bf \it  quarter edge}. Let $e=(v_1, v_2, \cdots, v_r)$ be a full edge, then $(v_j, v_{j+1})\; \forall\; 1\leq j\leq r-1$ is said to be a {\bf \it partial edge} of $e$, and if $e=(v_1, v_2, \cdots, v_{r-1}, v_r)$ is an half edge then $(v_1, v_2)$ is said to be partial half edge if $v_1$ is middle end vertex,  $(v_{r-1}, v_r)$ is a partial half edge if $v_r$ is middle end vertex and $(v_i, v_{i+1})\; \forall\; 2\leq i\leq r-2$ are partial edges. 
\vskip2mm
\begin{example}
Consider the vertex set  $V = \{w_1,w_2,w_3,w_4,w_5,w_6,w_{7}\}$ and an edge set
$E = \{(w_1,w_2,w_3,w_4),$ $(w_2,w_5,w_6),$ $(w_3,w_7,w_6),$ $(w_5,w_7)\}$, note that $G=(V, E)$ is a semigraph. 
\end{example}

\begin{figure}[h]
\centering
 \begin{tikzpicture}[yscale=0.5]
 \Vertex[size=0.2,  label=$w_1$, position=below, color=black]{A} 
  \Vertex[x=2.5, size=0.2,label=$w_2$, color=none, position=below]{B} 
  \Vertex[x=5,size=0.2,label=$w_3$,position=below, color=none]{C} 
   \Vertex[x=7, size=0.2, label=$w_4$,position=below, color=black]{D}
  \Vertex[x=2.5,y=2,size=0.2,label=$w_5$,position=left, color=none]{E}  
  \Vertex[x=2.5, y=4, size=0.2, label=$w_6$,position=above, color=black]{F}
  \Vertex[x=3.8,y=2,size=0.2,label=$w_7$,position=right, color=none]{G}  
  \Edge(A)(B) \Edge(B)(C) \Edge(C)(D)  \Edge(G)(F) \Edge(F)(E)  \Edge(C)(G)
  \draw[thick](2.4, 0.3)--(2.6, 0.3);
  \draw[thick](2.65, 1.8)--(2.65,2.2);
    \draw[thick](3.6, 1.8)--(3.6,2.15);
    \draw[thick](2.5, 1.8)--(2.5,0.3);
    \draw[thick](2.65,2)--(3.6,2);
\end{tikzpicture}
\caption{}
\label{fig:1}
 \end{figure}

In Fig.~\ref{fig:1}, vertices $w_1,w_4$, and $w_6$ are the pure end vertices; $w_3$ is the pure middle vertex; $w_2$,$w_5$ and $w_7$ are the middle end vertices. Further, $(w_1, w_2, w_3, w_4)$ is full edge, whereas $(w_2, w_5, w_6)$ is an half edge with only $(w_2, w_6)$ as a partial half edge. Note that $(w_5, w_7)$ is a quarter edge. 
 \begin{definition}  \label{def:2}
Let $G$ be a semigraph with vertex set $V$ and edge set $E$. If $\forall\, u,\,v\in V$ there exist $e_{i_1},\cdots, e_{i_p}\in E$ such that $u\in e_{i_1},\; v\in e_{i_p}$ and $|e_{i_{j}}\cap e_{i_{j+1}}|=1,\; \forall\; 1\leq j\leq p-1$ then $G$ is called connected semigraph.
 \end{definition}
In this paper, our semigraph $G=(V, E)$ is connected. Here, $n$ denotes the number of vertices, $m$ denotes the number of edges such that $m=m_1+m_2+m_3+m_4;$ where $m_1,\;m_2,\;m_3,\; m_4$ are number of full edges, quarter edges, half edges with one partial half edge, and half edges with two partial half edges respectively. If $G$ is a graph then $m_2=m_3=m_4=0$ and $m=m_1$.

\vskip2mm

 \subsection{Adjacency matrix}
  Consider a semigraph $G$=$(V, E),$ with $V=\{v_1, v_2,\cdots, v_n\}$ as a vertex set and $E=\{e_1, e_2,\cdots, e_m\}$ as an edge set.  
Recall that the graph skeleton $G^{S}$ ~\cite[definition 1.5]{pms1} of $G$ is the graph defined on $V$ such that two vertices are adjacent in $G^{S}$ if and only if they are consecutively adjacent in $G.$ For any two vertices $u_i$ and $u_j$ of an edge $e$, $d_{e}(u_i,\,u_j)$ represent the distance between $u_i$ and $u_j$ in the graph skeleton of the edge $e$. As each pair of vertices in semigraph belongs to at most one edge, the distance $d_{e}(u_i,u_j)$ is well-defined.  
  \begin{definition}~\cite{pms1} \label{def:4}
  We index the rows and columns of a matrix $A=(a_{ij})_{n\times n}$ by vertices $ v_1, v_2,\cdots, v_n,$ where $a_{ij}$ is given as follows:
$$a_{ij}=\begin{cases}
 d_{e}(v_i,v_j),& \text{if $v_i,\; v_{j}$ belong to a full edge or a half edge such that } \\
 &\text{$(v_i, v_j)$ is neither a partial half edge nor a quarter edge}\\
\;\;\; \frac{1}{2}, & \text{if $(v_i,\;  v_{j})$ is a partial half edge}\\
\;\;\;\frac{1}{4}, &\text{if $(v_i,  v_{j})$ is a quarter edge}\\
\;\;\;0,&\text{otherwise}
\end{cases}$$ 
\end{definition}
\noindent The above matrix $A=(a_{ij})_{n\times n}$ is the adjacency matrix of semigraph $G.$ 
Let $A_{i}$ be the $i^{th}$ row of the adjacency matrix $A$ associated with the vertex, say $v_i$; the number $d_i=A_{i}\mathbf{1}$ is the degree of $v_i$, where $\mathbf{1}$ is the column vector of all entries 1.
   
\section{Laplacian of Semigraph}
Let $G=(V,E)$ be a semigraph with $|V|=n,\; |E|=m$. The Laplacian of graph is extensively studied matrix \cite{mohar}; along similar lines, we define the Laplacian of semigraph. If $G$ is a graph, then semigraph Laplacian is the same as the graph Laplacian. This motivated us to study the spectral properties of Laplacian of semigraphs. Let $D$ is a diagonal degree matrix  and $A$ is the adjacency matrix of semigraph $G$, then we define the Laplacian as $$L=D-A$$
\begin{example} 
The Laplacian matrix of the semigraph in Fig.\ref{fig:1} is
 $$\begin{pmatrix} 
6&-1&-2&-3&0&0&0\\
-1&6.5&-1&-2&-\frac{1}{2}&2&0\\
-2&-1&7&-1&0&-2&-1\\
-3&-2&-1&6&0&0&0 \\
0&-\frac{1}{2}&0&0&1.75&-1&-\frac{1}{4}\\
0&-2&-2&0&-1&6&-1\\
0&0&-1&0&-\frac{1}{4}&-1&2.25
\end{pmatrix}    $$
\end{example}

Our goal is to study the eigenvalues of Laplacian matrix. The following result helps us to deduce that eigenvalues are non-negative.  
\subsection{Positive semi-definite}

\par Here, we show that the Laplacian of semigraph is positive semi-definite. This makes the study of spectra of semigraphs interesting. \\
Let $G$ be a semigraph, and let $e=(u, v)$ be an edge having two vertices, then we have three types of edges based on whether $u,\;v$ are pure end vertices or middle end vertices. Let $L_e$ denote the Laplacian matrix of the semigraph edge $e$. Our aim is to find a quadratic expression $x^tL_ex$ for the edge $e$ of size, say $r$, where $x\in \mathbb{R}^r$.
\begin{observation} \label{obs:0}
Let $e=(u, v)$ be an edge in $G$. For $x=(x_1, x_2)^t\in \mathbb{R}^2$, $$x^tL_{e}x=\mu (x_1-x_2)^2$$ 
where $\mu=\begin{cases} 
      1, & \text{if both u and v are pure end vertices} \\
      \frac{1}{2}, & \text{if one of u and v is a middle end vertex}  \\
      \frac{1}{4}, & \text{if both u and v are middle end vertices}
   \end{cases}$

\begin{itemize}
\item[Case 1]: Both $u$ and $v$ are pure end vertices.  Then $L_e=\begin{pmatrix} 1&-1\\-1&1\end{pmatrix}.$\\ 
Thus, for $x=\begin{pmatrix}x_1\\x_2\end{pmatrix}\in \mathbb{R}^2$, we get $x^tL_ex=(x_1-x_2)^2$.
\item[Case 2]: Both $u$ and $v$ are middle end vertices.  Then $L_e=  \renewcommand{\arraystretch}{1.5}\begin{pmatrix} \frac{1}{4}&-\frac{1}{4}\\-\frac{1}{4}&\frac{1}{4}\end{pmatrix}.$\\
Thus, for $x=\begin{pmatrix}x_1\\x_2\end{pmatrix}\in \mathbb{R}^2$, we get $x^tL_ex=\frac{1}{4}(x_1-x_2)^2$.
\item[Case 3]: One of $u$ and $v$ is middle end vertex.  Then $L_e=  \renewcommand{\arraystretch}{1.5}\begin{pmatrix} \frac{1}{2}&-\frac{1}{2}\\-\frac{1}{2}&\frac{1}{2}\end{pmatrix}.$ \\
Thus, for $x=\begin{pmatrix}x_1\\x_2\end{pmatrix}\in \mathbb{R}^2$, we get $x^tL_ex=\frac{1}{2}(x_1-x_2)^2$.
\end{itemize}
\end{observation}

We need the following lemma to get a quadratic expression for laplacian of an edge $|e|\geq 3$
\begin{lemma}\label{lemma:1}
Let $x\in \mathbb{R}^l$, where $x=\begin{pmatrix}x_1\\ \vdots \\ x_l\end{pmatrix}.$ Let $L_e$ denote the Laplacian of semigraph edge $e=(v_{i_1}, v_{i_2},\cdots,v_{i_l}),\; l\geq3$. Then  $$\displaystyle x^tL_e x=\sum_{j=1}^{l-1}\sum_{i=1}^{l-j}\mu_{ji}(x_j -x_{j+i})^2$$ 
where  $\mu_{ji}=i,$ except when $v_{i_1}$ is middle end vertex then $\mu_{11}=\frac{1}{2}$, if $v_{i_l}$ is middle end vertex then $\mu_{(l-1)1}=\frac{1}{2}$.
\end{lemma}
\begin{proof}
For an edge $e$, $L_e$ denote the Laplacian of semigraph edge. We prove this result by induction on $l$. 
\begin{itemize}
\item{Step 1:} Let $l=3$, $e=(v_{i_1}, v_{i_2}, v_{i_3})$. Then $$L_e=\begin{pmatrix}\mu_{11}+2&-\mu_{11}&-2\\ -\mu_{11}&\mu_{11}+\mu_{21}&-\mu_{21}\\-2&-\mu_{21}&\mu_{21}+2\end{pmatrix}$$
where $\mu_{11}=1$ if $v_{i_1}$ is pure end vertex and it is $\frac{1}{2}$ when $v_{i_1}$ is middle end vertex and $\mu_{21}=1$ if $v_{i_3}$ is pure end vertex and it is $\frac{1}{2}$ when $v_{i_3}$ is middle end vertex. Thus, for $x=\begin{pmatrix}x_1,x_2,x_3\end{pmatrix}^t\in \mathbb{R}^3$, we have 
\begin{align*}x^tL_ex=&\;(\mu_{11}+2)x^2_1-\mu_{11}x_1x_2-2x_1x_3-\mu_{11}x_1x_2+(\mu_{11}+\mu_{21})x^2_2-\mu_{21}x_2x_3\\
&-2x_1x_3-\mu_{21}x_2x_3+(\mu_{21}+2)x^2_3 \\
=&\;\mu_{11}(x^2_1-2x_1x_2+x^2_2)+2(x^2_1-2x_1x_3+x^2_3)+\mu_{21}(x^2_2-2x_2x_3+x^2_3)\\
=&\;\mu_{11}(x_1-x_2)^2+2(x_1-x_3)^2+\mu_{21}(x_2-x_3)^2\\
x^tL_ex=&\;\sum_{j=1}^{3-1}\sum_{i=1}^{3-j}\mu_{ji}(x_j -x_{j+i})^2
\end{align*}
\item{Step 2:} Assume that formula is true for an edge of size $l-1$. \\
For $x\in \mathbb{R}^{l-1}$, where $x=\begin{pmatrix}x_1,x_2,\cdots  x_{l-1}\end{pmatrix}^t$, we have 
$$\displaystyle x^tL_e x=\sum_{j=1}^{l-2}\;\sum_{i=1}^{l-1-j}\mu_{ji}(x_j -x_{j+i})^2$$
\item{Step 3:} Let $e=(v_{i_1}, v_{i_2},\cdots, v_{i_{l-1}},v_{i_l})$. The Laplacian matrix $L_e$ of the edge $e$ is as follows:
$$\begin{pmatrix}
d_1&-\mu_{11}&-2&-3&\cdots&-l+3&-l+2&-l+1\\
-\mu_{11}&d_2&-1&-2&\cdots&-l+4&-l+3&-l+2\\
-2&-1&d_3&-1&\cdots&-l+5&-l+4&-l+3 \\
-3&-2&-1&d_4&\cdots&-l+6&-l+5&-l+4 \\
&&&&\ddots&&&\\
-l+3&-l+4&-l+5&-l+6&\cdots&d_{l-2}&-1&-2\\
-l+2&-l+3&-l+4&-l+5&\cdots&-1&d_{l-1}&-\mu_{(l-1)1}\\
-l+1&-l+2&-l+3&-l+4&\cdots&-2&-\mu_{(l-1)1}&d_{l}
\end{pmatrix}$$
where, $d_j$ is the degree of the  $i_j^{th}$ vertex of $e$. \\
Here, $d_j=(j-1)+\cdots+2+1+1+2+\cdots+(l-j)$ for all $2< j<l-1$.  \\
And \begin{align*}d_1=&\;\mu_{11}+2+3+\cdots+(l-1)\\
d_2=&\;\mu_{11}+1+2+\cdots+(l-2)\\
d_{l-1}=&\;(l-2)+(l-3)+\cdots+1+\mu_{(l-1)1}\\
d_r=&\; (l-1)+(l-2)+\cdots+2+\mu_{(l-1)1}
\end{align*}
We can rewrite $L_e$ as 
\pagebreak
$$\begin{pmatrix}
d_1-(l-1)&-\mu_{11}&-2&\cdots&-l+3&-l+2&0\\
-\mu_{11}&d_2-(l-2)&-1&\cdots&-l+4&-l+3&0\\
-2&-1&d_3-(l-3)&\cdots&-l+5&-l+4&0 \\
&&&\ddots&&&\\
-l+3&-l+4&-l+5&\cdots&d_{l-2}-2&-1&0\\
-l+2&-l+3&-l+4&\cdots&-1&d_{l-1}-\mu_{(l-1)1}&0\\
0&0&0&\cdots&0&0&0
\end{pmatrix}$$
\begin{center}+\end{center}
$$\begin{pmatrix}
l-1&0&0&\cdots&0&0&-(l-1)\\
0&l-2&0&\cdots&0&0&-l+2\\
0&0&l-3&\cdots&0&0&-l+3 \\
&&&\ddots&&&\\
0&0&0&\cdots&-2&0&-2\\
0&0&0&\cdots&0&\mu_{(l-1)1}&-\mu_{(l-1)1}\\
-l+1&-l+2&-l+3&\cdots&-2&-\mu_{(l-1)1}&d_l
\end{pmatrix}$$
For the first matrix in the summand, say $B$, the last row and column are zero. Hence, the last column and row of $B$ don't contribute anything in the quadratic form expression. Thus, for $x=\begin{pmatrix}x_1,x_2,\cdots  x_{l}\end{pmatrix}^t \in \mathbb{R}^l$ we have $x^tBx=x^tL_{e'} x$, where $e'=(v_{i_1}, v_{i_2},\cdots, v_{i_{l-1}})$. 
Thus, by induction assumption, for $x\in \mathbb{R}^{l}$, with $x=\begin{pmatrix}x_1,x_2,\cdots  x_{l}\end{pmatrix}^t$, we have
$$\displaystyle x^tL_{e'} x=\sum_{j=1}^{l-2}\sum_{i=1}^{l-1-j}\mu_{ji}(x_j -x_{j+i})^2$$ 
 where $\mu_{ji}=i$ except when $v_1$ is middle end vertex then $\mu_{11}=\frac{1}{2}$, here without loss of generality we assume that $v_{i_{l-1}}$ is not a middle end vertex of the edge $e'$.\\
 Thus, for $x\in \mathbb{R}^{l}$, where $x=\begin{pmatrix}x_1,x_2,\cdots  x_{l}\end{pmatrix}^t$, we get
$$\displaystyle x^tL_{e} x=x^tL_{e'} x +x^tL_{e-e'}x\;\;\;\;\;\;\cdots (1)$$ 
where $L_{e-e'}$ is the second matrix in the summand. 

Note that the second matrix in the summand gives us  \begin{align*}x^tL_{e-e'}x=&(l-1)x^2_1-(l-1)x_1x_l+(l-2)x^2_2-(l-2)x_2x_l-\cdots\\
&+2x^2_{l-2}-2x_{l-2}x_l+\mu_{(l-1)1}x^2_{l-1}-\mu_{(l-1)1}x_{l-1}x_r-(l-1)x_1x_l\\
&-(l-2)x_2x_l-\cdots-2x_{l-2}x_l-\mu_{(l-1)1}x_{l-1}x_l+d_lx^2_l\\
&=(l-1)(x_1-x_r)^2+\cdots+2(x_{l-1}-x_l)^2+\mu_{(l-1)1}(x_{l-1}-x_l)^2
\end{align*}
\end{itemize}
Thus by (1), we get 
$$\displaystyle x^tL_e x=\sum_{j=1}^{r-1}\sum_{i=1}^{r-j}\mu_{ji}(x_j -x_{j+i})^2$$ 
\end{proof}

\begin{theorem}
Let $L$ be the Laplacian matrix of a semigraph $G=(V, E)$, $L$ is positive semi-definite.
\end{theorem}
\begin{proof}
For an edge $e=(v_{i_1}, v_{i_2},\cdots, v_{i_l})$ we re-write $L_e$ as 
\[\renewcommand{\arraystretch}{1.5} \begin{pmatrix}
  \begin{matrix}
  L_e
  \end{matrix}
  & \rvline & \bigzero_{l\times n-l} \\
\hline
  \bigzero_{n-l\times l} & \rvline &
  \begin{matrix}
  \bigzero_{n-l\times n-l}
 \end{matrix}
\end{pmatrix}_{n\times n} \]
By additivity, we can write $\displaystyle L=\sum_{e\in E}L_e$\\
Thus, $$\displaystyle x^tLx=x^t\left(\sum_{e\in E}L_e\right)x=\sum_{e\in E}x^tL_ex$$
By the Lemma \ref{lemma:1} and observation \ref{obs:0}, we get
$$\displaystyle x^tLx= \sum_{e\in E}\sum_{j=1}^{l-1}\sum_{i=1}^{l-j}\mu_{ji}(x_j -x_{j+i})^2$$
where $\mu_{ji}$ are defined as earlier, and these are positive numbers. This implies $L$ is positive semi-definite.
\end{proof}
\begin{remark}
We define the signless laplacian $Q$ as $D+A$. Using the similar arguments above, we get
$$\displaystyle x^tQx= \sum_{e\in E}\sum_{j=1}^{l-1}\sum_{i=1}^{l-j}\mu_{ji}(x_j +x_{j+i})^2$$
Hence, signless laplacian is also positive semi-definite.
\end{remark}

\subsection{Algebraic connectivity of semigraph}
We show that the graph theory result about algebraic connectivity holds true for semigraphs, and the proof goes similar as well. Recall that for any graph $G$, graph is connected iff second eigenvalue $\lambda_2$ of Laplacian is positive. 
\begin{theorem}
Let $G=(V,E)$ be a semigraph, and let $0=\lambda_1\leq\lambda_2\leq\lambda_3\leq\cdots\leq \lambda_n $.  Then $G$ is connected iff  $\lambda_2 >0$.
\end{theorem}
\begin{proof}
Assume that $G$ is disconnected. We show that $\lambda_2=0$.\\
For simplicity, assume that $G_1$ and $G_2$ are two connected components of $G$. By reordering the indices, we can re-write the Laplacian matrix as follows: 
$$L=\begin{pmatrix}   L_{G_1}&\mathbf{0}\\\mathbf{0}&L_{G_2}
\end{pmatrix}$$
Observe that $\begin{pmatrix}   \mathbf{1}\\ \mathbf{0}
\end{pmatrix},\;\begin{pmatrix}   \mathbf{0}\\\mathbf{1}
\end{pmatrix}$ are two orthogonal eigenvectors of the eigenvalue 0. Thus, the geometric multiplicity of 0 is greater than or equal to 2 and hence algebraic multiplicity is greater then or equal to 2 implies $\lambda_2=0$.\\
Assume that $G$ is connected; we show that algebraic multiplicity of 0 is 1. \\
Let $f\in\mathbb{R}^n$ be an eigenvector of 0. Thus, $Lf=0$, hence by Lemma \ref{lemma:1}, for each edge we get
$$\displaystyle 0=f^tLf= \sum_{j=1}^{r-1}\sum_{i=1}^{r-j}\mu_{ji}(x_j -x_{j+i})^2$$
Therefore,  $x_1=x_j$ for all vertices in an edge $e$. Thus, $f$ is constant on each edge, and since $G$ is connected $f$ is constant on $V$. Thus, $f$ is a constant vector; hence geometric multiplicity of 0 is 1. Thus, algebraic multiplicity is 1. Hence, $\lambda_2>0$. 
\end{proof}

\section{Bounds on the largest Laplacian eigenvalue}
We know that Laplacian is positive semi-definite, suppose $0= \lambda_1 \leq \lambda_2 \leq \cdots \leq \lambda_n$ are the eigenvalues. Let $\Delta$ be the maximum degree of a vertex, in this section, we prove that $\lambda_n \geq \Delta +1$ 
 \begin{theorem}
 Let $G=(V, E)$ be a semigraph with at least one edge. Let $0= \lambda_1 \leq \lambda_2 \leq \cdots \leq \lambda_n$ then  $\lambda_n \geq \Delta +1$, where $\Delta$ is the largest degree of vertex in the semigraph. 
 \end{theorem}
 \begin{proof}
 As $L=(l_{ij})$ is positive semi-definite, by Cholsky decomposition $L=T^tT$, where $T=(t_{ij})$ is the lower triangular matrix with non-negative diagonal entries. WLOG, assume that $d_1$, degree of $v_1$ is the largest degree. Thus, by comparison we get $d_1=l_{11}=t^2_{11} \implies t_{11}=\sqrt{d_1}.$ Comparing the entries of the first columns we get $l_{i1}=t_{11}t_{i1}\; \forall i=1,2,\cdots, n.$ Hence, $$l_{i1}=\sqrt{d_1}\;t_{i1}\;\cdots (1)$$
 Now, the first diagonal entry of $T^tT$ equal is to 
 \begin{align*}
 \displaystyle \sum_{i=1}^{n}t^2_{i1}=&\;\sum_{i=1}^{n}\left(\frac{l_{i1}}{\sqrt{d_1}}\right)^2\\
 =&\; \frac{1}{d_1}\sum_{i=1}^{n}l^2_{i1}\\
 \sum_{i=1}^{n}t^2_{i1} =&\;\frac{1}{d_1}\left(d^2_1+\sum_{i=2}^{n}l^2_{i1}\right)\\
 =&\; d_1+\frac{1}{d_1}\sum_{i=2}^{n}l^2_{i1}
 \end{align*}
Note that $\displaystyle \sum_{i=2}^{n}l^2_{i1} \geq \sum_{i=2}^{n}|l_{i1}|$ (Here, we get equality when $G$ is a graph).
$$\implies \sum_{i=1}^{n}t^2_{i1} \geq d_1+\frac{1}{d_1}\sum_{i=2}^{n}|l_{i1}|$$
Note that $\displaystyle \sum_{i=2}^{n}|l_{i1}|$ is degree $d_1$ of the first vertex $v_1$.
Thus, we have $$\sum_{i=1}^{n}t^2_{i1} \geq d_1+1$$
The largest eigenvalue of $T^tT$ exceeds or equals the largest diagonal entry of $T^tT$; hence it exceeds or equals $d_1+1$. As eigenvalues of $L,\; T^tT,\; TT^t$ are same implies $\lambda_n \geq \Delta +1$, where $d_1=\Delta$.
 \end{proof} 
 \begin{remark}
 The proof is similar to the proof of \cite[Theorem 4.12]{rbp}. In fact, when $G$ is a graph, we get the graph theory result as a consequence of the above proof.  
 \end{remark}
 If $i^{th}$ and $j^{th}$ vertices make a partial edge, we denote it by $j \sim_{S} i$, and if they make a partial half edge, then we denote it by $j \sim_{|-} i$. If $i^{th}$ and $j^{th}$ vertices make a quarter edge then we denote it by $j \sim_{|-|} i$, and if $d_{e}(i, j)=l$ for some edge then we denote it by $j \sim_{l} i$. Let $d^S_i$ is the degree contribution to degree of $i^{th}$ vertex due to all partial edges incident to it, $d^{\frac{1}{2}}_i$ is the degree contribution to degree of $i^{th}$ vertex due to all partial half edges incident to it, $d^{\frac{1}{4}}_i$ is the degree contribution to degree of $i^{th}$ vertex due to all quarter edges incident to it, $d^{l}_i$ is the degree contribution to degree of $i^{th}$ vertex due to all adjacent vertices which are $l$ distance apart from it. Note that $\displaystyle d_i= d^S_i+d^{\frac{1}{2}}_i+d^{\frac{1}{4}}_i+\sum_{l=2}^{r-1} d^{l}_i$.
 
Let $C_S(i, j)$ be the number of vertices that are consecutively adjacent to both $i^{th}$ and $j^{th}$ vertices and form partial edges, $C_{\frac{1}{2}}(i, j)$ is the number of  vertices that are consecutively adjacent to both $i^{th}$ and $j^{th}$ vertices and form partial half edges,  $C_{\frac{1}{4}}(i, j)$ is the number of  vertices that are consecutively adjacent to both $i^{th}$ and $j^{th}$ vertices and form quarter edges, $C_l(i, j)$ is the number of vertices that are adjacent to both $i^{th}$ and $j^{th}$ vertices and are $l$ distance apart from each of them. Let $\displaystyle C(i, j)=C_S(i,j)+C_{\frac{1}{2}}(i,j)+C_{\frac{1}{4}}(i,j)+\sum_{l=2}^{r-1}lC_{l}(i,j)$. 
\vskip 1cm

 \begin{theorem}
Let $G=(V, E)$ be a semigraph with at least one edge. Let $0= \lambda_1 \leq \lambda_2 \leq \cdots \leq \lambda_n$ then  
$$\lambda_n \leq max\{d_i+d_j -C(i,j)\}$$
\end{theorem}
 \begin{proof}
 Let $\lambda_n$ be the largest eigenvalue. Let $x=(x_1,x_2,\cdots, x_n)^t\in \mathbb{R}^n$ be an eigenvector such that $Lx=\lambda_nx$. We choose $i$ such that $\displaystyle x_i=\underset{\;\;\;k}max\; x_k$. Further, we choose $j$ such that $\displaystyle x_j=\underset{\;\;\;k}min\{x_k: k\sim i\}.$ Here adjacency is the semigraph adjacency. Comparing the $i^{th}$, $j^{th}$ components of $(D-A)x=\lambda_nx$, we get
 $$\displaystyle \lambda_nx_i=d_ix_i-\sum_{k \sim_{s} i}x_{k} -\frac{1}{2}\sum_{k \sim_{|-} i}x_{k} -\frac{1}{4}\sum_{k \sim_{|-|} i}x_{k} -\sum_{l=2}^{r-1}\sum_{k \sim_{l} i}lx_{k}$$
 and
  $$\displaystyle \lambda_nx_j=d_jx_j-\sum_{k \sim_{s} j}x_{k} -\frac{1}{2}\sum_{k \sim_{|-} j}x_{k} -\frac{1}{4}\sum_{k \sim_{|-|} j}x_{k} -\sum_{l=2}^{r-1}\sum_{k \sim_{l} j}lx_{k}$$
  Rewriting the above two equations, we get
  
   \begin{align*} 
  \displaystyle \lambda_nx_i=&\;d_ix_i-\sum_{k \sim_{s} i; \;k\sim_{s} j}x_{k}-\sum_{k \sim_{s} i; \;k\nsim_{s} j}x_{k} -\frac{1}{2}\sum_{k \sim_{|-} i;\;k\sim_{|-}j}x_{k}-\frac{1}{2}\sum_{k \sim_{|-} i;\;k\nsim_{|-}j}x_{k}\\
  &  -\frac{1}{4}\sum_{k \sim_{|-|} i;\;k\sim_{|-|}j} x_{k}-\frac{1}{4}\sum_{k \sim_{|-|} i;\;k\nsim_{|-|}j} x_{k} -\sum_{l=2}^{r-1}\sum_{k \sim_{l} i;\;k\sim_{l}j}lx_{k}-\sum_{l=2}^{r-1}\sum_{k \sim_{l} i;\;k\nsim_{l}j}lx_{k}
  \end{align*}
  and
  \begin{align*} 
  \displaystyle \lambda_nx_j=&\;d_jx_j-\sum_{k \sim_{s} j; \;k\sim_{s} i}x_{k}-\sum_{k \sim_{s} j; \;k\nsim_{s} i}x_{k} -\frac{1}{2}\sum_{k \sim_{|-} j;\;k\sim_{|-}i}x_{k}-\frac{1}{2}\sum_{k \sim_{|-} j;\;k\nsim_{|-}i}x_{k}\\
  &  -\frac{1}{4}\sum_{k \sim_{|-|} j;\;k\sim_{|-|}i} x_{k}-\frac{1}{4}\sum_{k \sim_{|-|} j;\;k\nsim_{|-|}i} x_{k} -\sum_{l=2}^{r-1}\sum_{k \sim_{l} j;\;k\sim_{l}i}lx_{k}-\sum_{l=2}^{r-1}\sum_{k \sim_{l} j;\;k\nsim_{l}i}lx_{k}
  \end{align*}
Subtract the second equation from the first
 \begin{align*} 
  \displaystyle \lambda_n(x_i- x_j)=&d_ix_i-\;d_jx_j-\sum_{k \sim_{s} i; \;k\nsim_{s} j}x_{k}+\sum_{k \sim_{s} j; \;k\nsim_{s} i}x_{k} -\frac{1}{2}\sum_{k \sim_{|-} i;\;k\nsim_{|-}j}x_{k}+\frac{1}{2}\sum_{k \sim_{|-} j;\;k\nsim_{|-}i}x_{k}\\
  &  -\frac{1}{4}\sum_{k \sim_{|-|} i;\;k\nsim_{|-|}j} x_{k}+\frac{1}{4}\sum_{k \sim_{|-|} j;\;k\nsim_{|-|}i} x_{k} -\sum_{l=2}^{r-1}\sum_{k \sim_{l} i;\;k\nsim_{l}j}lx_{k}+\sum_{l=2}^{r-1}\sum_{k \sim_{l} j;\;k\nsim_{l}i}lx_{k} \\
 \leq &\;\;d_ix_i-d_jx_j- \left(d^S_i-C_S(i,j)\right)x_j+\left(d^S_j-C_S(i,j)\right)x_i-\left(d^{\frac{1}{2}}_i-\frac{1}{2}C_{\frac{1}{2}}(i,j)\right)x_j\\
 &+\left(d^{\frac{1}{2}}_j-\frac{1}{2}C_{\frac{1}{2}}(i,j)\right)x_i-\left(d^{\frac{1}{4}}_i-\frac{1}{4}C_{\frac{1}{4}}(i,j)\right)x_j+\left(d^{\frac{1}{4}}_j-\frac{1}{4}C_{\frac{1}{4}}(i,j)\right)x_i\\
 &-\left(\sum_{l=2}^{r-1}d^{l}_j-lC_{l}(i,j)\right)x_j+\left(\sum_{l=2}^{r-1}d^{l}_i-lC_l(i,j)\right)x_i 
 \end{align*}
 By simplifying and rewriting it, we get
  \begin{align*} 
  \displaystyle \lambda_n(x_i- x_j)
 \leq &\;\;(d_i+d_j)(x_i-x_j)- C_S(i,j)(x_i-x_j)-\frac{1}{2}C_{\frac{1}{2}}(i,j)(x_i-x_j)\\
 &-\frac{1}{4}C_{\frac{1}{4}}(i,j)(x_i-x_j)-\left(\sum_{l=2}^{r-1}lC_{l}(i,j)\right)(x_i-x_j). 
 \end{align*}
 Thus, we get
 $$\lambda_n(x_i- x_j)\leq \;\; (d_i+d_j -C(i,j))(x_i-x_j)$$
 Note that $\lambda_n$ is not $0$ as our graph is connected having at least of one edge. Thus, there is some $j$ such that $j^{th}$ vertex is adjacent to $i^{th}$ vertex and $x_i>x_j$. Hence, $\lambda_n \leq (d_i+d_j-C(i,j))$ 
\end{proof}
\begin{remark}
 The proof is similar to the proof of \cite[Theorem 4.13]{rbp}, and if $G$ is a graph, we get the graph theory result as the consequence of above result. 
 \end{remark}

 \section{Laplacian Eigenvalues of stars}
 In this section, we define different types of star semigraphs and study their Laplacian eigenvalues. Consider a semigraph $S^3_{2, n}$ on $n+3$ vertices; all edges are of size two except one edge, which has three vertices, as shown in Fig. \ref{fig:2}
\begin{figure}[h]
\centering
  \begin{tikzpicture}[yscale=0.5]
 \Vertex[size=0.2, y=3, label=$v_2$, position=above, color=black]{B} 
 \Vertex[size=0.2,  label=$v_1$, position=180, color=none]{A} 
 \Vertex[size=0.2,y=-3,  label=$v_3$, position=below, color=black]{C} 
 \Edge(A)(B) \Edge(A)(C)
  \Vertex[size=0.2,x=1.1, y=2.5,  label=$v_4$, position=45, color=black]{D} 
    \Vertex[size=0.2,x=1.3, y=-3, label=$v_r$, position=right, color=black]{E} 
  \Vertex[size=0.2,x=-1.6, y=2.2, label=$v_{n+3}$, position=left, color=black]{F} 
  \Vertex[size=0.2,x=-1.2, y=-3.15, label=$v_{r+1}$, position=left, color=black]{G} 
   \draw[thick](0.2,0.2)--(1.1,2.45);
    \draw[thick](0.15,-0.3)--(1.15,-2.7);
    \draw[thick](-0.2,0.3)--(-1.5, 2.1);
    \draw[thick](-0.2,-0.3)--(-1.2,-3.1);
    \draw[thick](0.2, 0)--(0.2, 0.5);
   \draw[thick](0.15, -0.15)--(0.15, -0.6);
    \draw[thick](-0.2, 0)--(-0.2, 0.5);
   \draw[thick](-0.2, -0.1)--(-0.15, -0.6);

  \Edge[bend =20, style={dashed}](D)(E)
   \Edge[bend =20, style={dashed}](G)(F)

\end{tikzpicture}\\
$\text{Star semigraphs:}\;\;S^3_{2,n} $
\caption{ }
\label{fig:2}
\end{figure}

 The Laplacian $L$ is as follows: 
  \[
     \bordermatrix{ & {v_1} & {v_2} & {v_3}& {v_4} & {v_5} & \cdots & {v_{n+3}} \cr
       v_1 & \frac{(n+4)}{2}&-1&-1&-\frac{1}{2}&-\frac{1}{2}&\cdots&-\frac{1}{2} \cr
       v_2 & -1&3&-2&0&0&\cdots&0\cr
       v_3 & -1&-2&3&0&0&\cdots&0 \cr
        v_4 & -\frac{1}{2}&0&0&\frac{1}{2}&0&\cdots&0 \cr
       v_5 & -\frac{1}{2}&0&0&0&\frac{1}{2}&\cdots&0 \cr
       \vdots & \vdots&\vdots&\vdots&\vdots&\vdots&\ddots&\vdots\cr
       v_{n+3} &-\frac{1}{2}&0&0&0&0&\cdots&\frac{1}{2}} \qquad
 \]

 \begin{lemma}
The Laplacian spectrum of star semigraph $S^3_{2,n}$ is:
$$\begin{pmatrix}0&0.5&\lambda_1& \lambda_2& \lambda_3 \\ 1&n-1&1&1&1\end{pmatrix}$$ 
 where $\lambda_{1}, \lambda_{2},\lambda_{3}$ are roots of the cubic polynomial $\lambda^3 -\frac{n+17}{2}\;\lambda^2+(19+3n)\lambda -\frac{5n+15}{2}.$
\end{lemma}
\begin{proof}
Let $P_{n}(\lambda) =det(\lambda I -L)$ be the characteristic polynomial of $S^3_{2,n}$.
$$P_{n}(\lambda)=\begin{vmatrix}
\lambda-\frac{n+4}{2}&1&1&\frac{1}{2}&\cdots&\frac{1}{2}&\frac{1}{2} \\
1&\lambda-3&2&0&\cdots&0&0 \\
1&2&\lambda-3&0&\cdots&0&0\\
\frac{1}{2}&0&0&\lambda-\frac{1}{2}&\cdots&0&0 \\
 \vdots&\vdots&\vdots&\vdots&\ddots&\vdots&\vdots  \\ 
 \frac{1}{2}&0&0&0&\cdots&\lambda-\frac{1}{2}&0\\
 \frac{1}{2}&0&0&0&\cdots&0&\lambda-\frac{1}{2}
\end{vmatrix} $$

Using cofactor expansion along the first column, we get

\begin{align*}
P_{n}(\lambda)=&\left(\lambda-\frac{1}{2}\right)^n\left[\left((\lambda-3)^2-4\right)\left(\lambda-\frac{n+4}{2}\right)\right]-2\left(\lambda-\frac{1}{2}\right)^n(\lambda-5)\\
&-\frac{n}{4}\left(\lambda-\frac{1}{2}\right)^{n-1}\left[(\lambda-3)^2-4\right]
\end{align*}
$$P_{n}(\lambda)=\left(\lambda-\frac{1}{2}\right)^{n-1}\left[\left(\lambda^2-6\lambda +5\right)\left[\left(\lambda-\frac{1}{2}\right)\left(\lambda-\frac{n+4}{2}\right)-\frac{n}{4}\right]-2\left(\lambda-\frac{1}{2}\right)(\lambda-5)\right]$$
Simplifying it further, we get
$$P_{n}(\lambda)=\displaystyle \lambda \left(\lambda-\frac{1}{2}\right)^{n-1}\left(\lambda^3 -\frac{n+17}{2}\;\lambda^2+(19+3n)\lambda -\frac{5n+15}{2}\right)$$

Thus, the Laplacian spectra of star semigraphs are:$\left\{0, \frac{1}{2}\text{(n-1 times)},\;\lambda_{1}, \lambda_{2},\lambda_{3}\right\};$ where $\lambda_{1}, \lambda_{2},\lambda_{3}$ are roots of the cubic polynomial $$\lambda^3 -\frac{n+17}{2}\;\lambda^2+(19+3n)\lambda -\frac{5n+15}{2}.$$
\end{proof}

 \subsection{Laplacian eigenvalues of rooted 3-uniform semigraph tree}
 Here, we compute the Laplacian eigenvalues of the rooted 3-uniform semigraph tree. Let $T^3_n$ denote the semigraph on $2n+1$ vertices with $n$ edges. The edge set $E$ is given by $\{(v_1, v_{2i},v_{2i+1})\; \big{|}\; 1\leq i\leq n \}$
\begin{figure}[h]
\centering
  \begin{tikzpicture}[yscale=0.5]
 \Vertex[size=0.2, label=$v_1$, position=above, color=black]{A} 
  \Vertex[size=0.2, x=1, y=-1,label=$v_{2n}$, position=right, color=none]{B} 
   \Vertex[size=0.2, x=2, y=-2,label=$v_{2n+1}$, position=below, color=black]{C}
   \Vertex[size=0.2, x=-1, y=-1,label=$v_2$, position=left, color=none]{D} 
   \Vertex[size=0.2, x=-2, y=-2,label=$v_3$, position=below, color=black]{E}
  \Vertex[size=0.2, x=0, y=-2,label=$v_{2r}$, position=left, color=none]{F} 
   \Vertex[size=0.2, x=0, y=-4,label=$v_{2r+1}$, position=below, color=black]{G} 
   \Edge(A)(D) \Edge(D)(E) \Edge(A)(F) \Edge(F)(G) \Edge(A)(B) \Edge(B)(C)
  \Edge[bend =-20, style={dashed}](E)(G)
  \Edge[bend =-20, style={dashed}](G)(C)
\end{tikzpicture}\\
$T^3_{n} $ 
\caption{}
\label{fig:}
\end{figure}

The Laplacian matrix $L$ of 3-uniform tree: $T^{3}_{n}$ is
\[
     \bordermatrix{ & {v_1} & {v_2} & {v_3}& {v_4} & {v_5}  & \cdots  & {v_{2n-2}}& {v_{2n-1}} &{v_{2n}}& {v_{2n+1}} \cr
       v_1 & 3n&1&2&1&2&\cdots&1&2&1&2 \cr
       v_2 & 1&2&1&0&0&\cdots&0&0&0&0\cr
       v_3 & 2&1&3&0&0&\cdots&0&0&0&0 \cr
        v_4 & 1&0&0&2&1&\cdots&0&0&0&0 \cr
       v_5 & 2&0&0&1&3&\cdots&0&0&0&0 \cr
   \vdots & \vdots&\vdots&\vdots&\vdots&\vdots&\ddots&\vdots&\vdots&\vdots&\vdots\cr
       v_{2n-2} &1&0&0&0&0&\cdots&2&1&0&0 \cr
        v_{2n-1} &2&0&0&0&0&\cdots&1&3&0&0 \cr
       v_{2n} &1&0&0&0&0&\cdots&0&0&2&1 \cr
       v_{2n+1} &2&0&0&0&0&\cdots&0&0&1&3} \qquad
 \]

\begin{lemma}
The Laplacian spectrum of star semigraph $T^3_{ n}$ is:
$$\begin{pmatrix}0&\frac{5\pm\sqrt{5}}{2}&\lambda_1& \lambda_2 \\ 1&n-1&1&1\end{pmatrix}$$ 
 where $\lambda_{1}, \lambda_{2}$ are roots of the quadratic polynomial $\lambda^2-(3n+5)\lambda+10n+5.$
\end{lemma}
\begin{proof}
Let $T_{n}(\lambda)  =det(\lambda I -L)$ be the characteristic polynomial of $T^3_{ n}$.
$$T_{n}(\lambda)=\begin{vmatrix}
\lambda-3n&1&2&1&2&\cdots& 1&2&1&2 \\
1&\lambda-2&1&0&0&\cdots & 0&0&0&0 \\
2&1&\lambda-3&0&0 &\cdots & 0&0&0&0 \\
1&0&0&\lambda-2&1 &\cdots &0&0&0&0 \\
2&0&0&1&\lambda-3 & \cdots &0&0&0&0 \\
&&&\vdots&&\vdots&&&& \\
1&0&0&0&0 & \cdots & \lambda-2&1&0&0 \\
2&0&0&0&0 & \cdots &1&\lambda-3&0&0 \\
1&0&0&0&0&\cdots& 0&0&\lambda-2&1 \\
2&0&0&0&0&\cdots&0&0&1&\lambda-3
\end{vmatrix} $$
Using co-factor expansion along the first column, we get
\begin{align*}
T_{n}(\lambda)=&(\lambda-3n)\left[(\lambda-2)(\lambda-3)-1\right]^n-(\lambda-5)\left[(\lambda-2)(\lambda-3)-1\right]^{n-1}\\
+&2(-2\lambda+5)\left[(\lambda-2)(\lambda-3)-1\right]^{n-1}+\cdots+\\
&-(\lambda-5)\left[(\lambda-2)(\lambda-3)-1\right]^{n-1}+2(-2\lambda+5)\left[(\lambda-2)(\lambda-3)-1\right]^{n-1}\\
=&(\lambda-3n)\left[(\lambda-2)(\lambda-3)-1\right]^n-n(\lambda-5)\left[(\lambda-2)(\lambda-3)-1\right]^{n-1}\\
&+2n(-2\lambda+5)\left[(\lambda-2)(\lambda-3)-1\right]^{n-1}
\end{align*}
Simplifying this further, we get
$$T_{n}(\lambda)=\displaystyle \lambda \left(\lambda^2-5\lambda+5\right)^{n-1}\left(\lambda^2-(3n+5)\lambda+10n+5\right)$$
Thus, the Laplacian spectra of star semigraphs are:$\left\{0, \frac{5\pm\sqrt{5}}{2}\text{(n-1 times)},\;\lambda_{1}, \lambda_{2}\right\};$ where $\lambda_{1}, \lambda_{2}$ are roots of the quadratic polynomial $$\lambda^2-(3n+5)\lambda+10n+5.$$
\end{proof}
\section{Conclusion}
We could provide evidence that shows that spectral theory for semigraphs generalizes the spectral theory of graphs. The Laplacian matrix is symmetric positive semi-definite which is a very well understood class of matrices, and hence it opens doors to several research problems in spectral semigraph theory. \\

\noindent \textbf{Author Contributions:} I declare that I carried out all work on my own.\\
\textbf{Funding:} No funding.\\
\textbf{Availability of data and material} All data generated or analyzed during this study are included in this published
article.
\subsection*{Declarations:}
There is no conflict of interest.
\bibliographystyle{amsplain}

\end{document}